\documentclass[a4paper]{amsart} 
\usepackage{amsfonts,amsmath,amssymb,amsthm,amsxtra,amsopn,bbm}
\usepackage{hyperref}
\usepackage{paralist}
\usepackage{booktabs}
\usepackage{color}
\usepackage[all]{xy}
\usepackage{mathtools}

\newtheorem{theorem}{Theorem}[section]
\newtheorem{lemma}[theorem]{Lemma}
\newtheorem{proposition}[theorem]{Proposition}
\newtheorem{corollary}[theorem]{Corollary}

\theoremstyle{definition}
\newtheorem{definition}[theorem]{Definition}

\newtheorem{remark}[theorem]{Remark}

\numberwithin{equation}{section}

\DeclareMathOperator{\sgn}{sign}
\newcommand{\Mod}{\operatorname{\mathsf{Mod}}\nolimits}
\renewcommand{\mod}{\operatorname{\mathsf{mod}}\nolimits}
\newcommand{\End}{\operatorname{End}\nolimits}
\newcommand{\Hom}{\operatorname{Hom}\nolimits}

\newcommand{\Ext}{\operatorname{Ext}\nolimits}
\newcommand{\rad}{\operatorname{rad}\nolimits}

\newcommand{\Fun}{\operatorname{Fun}\nolimits}
\newcommand{\Rep}{\operatorname{\mathsf{Rep}}\nolimits}
\newcommand{\rep}{\operatorname{\mathsf{rep}}\nolimits}
\newcommand{\add}{\operatorname{\mathsf{add}}\nolimits}

\newcommand{\id}{\mathrm{id}}

\def\frS{{\mathfrak S}}

\def\Mk{{\mathsf M}_k}
\def\F{{\mathcal F}}
\def\G{{\mathcal G}}
\def\I{{\mathbbm 1}}

\def\bbN{{\mathbb N}}

\def\P{{\mathsf P}}

\def\op{{\mathrm {op}}}

\newcommand{\HOM}{\operatorname{\mathcal{H}\!\!\:{\it om}}\nolimits}
\def\Filt{{\mathrm{Filt}}}

\begin{document}

\title{The adjoints of the Schur functor}
\author{Rebecca Reischuk}
\address{Rebecca Reischuk, Fakult\"at f\"ur Mathematik, Universit\"at Bielefeld, D-33501 Bielefeld, Germany.}
\email{rreischuk@math.uni-bielefeld.de}

\begin{abstract}
We show that the left and right adjoint of the Schur functor can 
be expressed in terms of the monoidal structure of strict polynomial functors.
Using this result we give a necessary and sufficient condition for when the tensor product of two simple strict
polynomial functors is again simple.
\end{abstract}

\maketitle
\setcounter{tocdepth}{1}
\tableofcontents

\section{Introduction}
In his dissertation, Issai Schur defines an algebra, nowadays known as the Schur algebra,
whose module category is equivalent to polynomial representations
of the general linear group. He then
uses a functor, now called the Schur functor, to relate representations
of the general linear group and representations of the symmetric group. 

For suitable choices of parameters, another category, namely the category of strict polynomial functors,
is equivalent to the category of modules over the Schur algebra.
This category, first defined by Friedlander and Suslin in \cite{FS1997},
inherits a tensor product from the category of divided powers.
A tensor product for the category of representations of the symmetric group is given by its Hopf algebra structure.
It has been shown recently that the Schur functor preserves this monoidal structure (\cite{AR2015}). 

The Schur functor has fully faithful left and right adjoints. These adjoints have been studied in order to relate 
the cohomolgy of general linear and symmetric groups (cf.\ \cite{DEN2004}) and to relate (dual) Specht filtrations of symmetric group
modules to Weyl filtrations of modules over the general linear group (cf.\ \cite{HN2004}).

\medskip

We show that the left resp.\ right adjoint of the Schur functor can be expressed in terms of the internal tensor
product resp.\ internal hom of strict polynomial functors (Theorem~\ref{Th:GotimesF} resp.\ Theorem~\ref{Th:GHomF}).
We make use of these expressions to relate the two adjoints.
In addition, we will see that the adjoints induce equivalences of categories when restricting to the subcategories of injective resp.\ projective
strict polynomial functors.

In the last section we consider the tensor product of two simple strict polynomial functors. Touz\'e
showed in \cite{Tou2015} that in almost all cases such a tensor product is not simple. We use the left adjoint of the Schur functor
to calculate the remaining cases. As a consequence we get  a necessary and sufficient condition in terms
of $\Ext$-vanishing between certain simple functors for when the tensor
product of two  simple strict polynomial functors is simple (Theorem~\ref{Th:simples}).
In the case $n=d=p$ a full characterization is given (Theorem~\ref{n=d=p}).

\subsection*{Acknowledgements} 
I would like to thank Karin Erdmann for valuable comments and discussions about
representations of the symmetric group during a research visit in Oxford. 
In particular the results in the last section were completed with her assistance.
I am very grateful to Greg Stevenson for many fruitful discussions and 
his continuous advice on (monoidal) categories.
I am thankful to Nicholas Kuhn for comments on an earlier version of this paper.

\section{Strict polynomial functors}
In the following we briefly recall the definitions of strict polynomial functors and of the internal tensor product as described in \cite[Section 2]{Kr2013}.
Let $k$ be a commutative ring and denote by $\P_k$ the category of finitely generated projective $k$-modules. 
For $d\in\bbN$ denote by $\frS_d$ the symmetric group permuting $d$ elements and for
$V\in\P_k$ let $\frS_d$ act on the right on  $V^{\otimes d}$  by permuting the factors.

\subsection*{Divided, symmetric and exterior powers}
The submodule $\Gamma^d V$ consisting of the $\frS_d$-invariant part of $V^{\otimes d}$ is called the module of \emph{divided powers}.
The coinvariant part is the module of \emph{symmetric powers}, denoted by $S^d V$. The quotient of $V^{\otimes d}$  by the ideal
generated by $v\otimes v$ are the \emph{exterior powers} $\Lambda^dV$.
Since the $k$-modules $\Gamma^d V$, $S^d V$ and $\Lambda^d V$ are free provided $V$ is free, 
sending a module $V$ to $\Gamma^d V$, $S^d V$ resp.\ $\Lambda^d V$ yields functors $\Gamma^d, S^d, \Lambda^d\colon\P_k\to\P_k$.

\subsection*{The category of degree $d$ divided powers} We define the category $\Gamma^d\P_k$ to be the category with
the same objects as $\P_k$ and where the morphisms between two objects $V$ and $W$ are given by
\[\Hom_{\Gamma^d\P_k}(V,W):=\Gamma^d\Hom(V,W)=(\Hom(V,W)^{\otimes d})^{\frS_d}.\]

\subsection*{The category of strict polynomial functors} Finally we define $\Rep\Gamma^d_k$ to be
the  category of $k$-linear representations of $\Gamma^d\P_k$, i.e.\
\[\Rep\Gamma^d_k=\Fun_k(\Gamma^d\P_k,\Mk),\]
where $\Mk$ denotes the category of all $k$-modules. 
The morphisms between two strict polynomial functors $X,Y$ are denoted by $\Hom_{\Gamma^d_k}(X,Y).$

The full subcategory of \emph{finite representations}, i.e.\ $X\in\Rep\Gamma^d_k$ such that $X(V)\in\P_k$ for all $V\in\Gamma^d\P_k$, is denoted
by $\rep\Gamma^d_k$.

The strict polynomial functor \emph{represented by $V\in\Gamma^d\P_k$} is given by
\[\Gamma^{d,V}:=\Hom_{\Gamma^d\P_k}(V,-).\]
For $X\in\Rep\Gamma^d_k$, the Yoneda isomorphism yields
\begin{align}\label{Yoneda}\Hom_{\Gamma^d_k}(\Gamma^{d,V},X)\cong X(V).\end{align}
We have $\Gamma^{d,k}=\Hom_{\Gamma^d\P_k}(k,-)=\Gamma^d\Hom(k,-)\cong\Gamma^d(-)$ and thus $\Gamma^d\in\Rep\Gamma^d_k$.
It is not hard to see that also $S^d, \Lambda^d\in\Rep\Gamma^d_k$.

\subsection*{External tensor product}
For non-negative integers $d,e$ and $X\in\Rep\Gamma^d_k$ and $Y\in\Rep\Gamma^e_k$ we can form the external tensor product 
\begin{align*}
 X\boxtimes Y\in \Rep\Gamma^{d+e}_k.
\end{align*}
 It is given on objects by $(X\boxtimes Y)(V)=X(V)\otimes Y(V)$ and on morphisms via the map
\[\Gamma^{d+e}\Hom(V,W)\to\Gamma^d\Hom(V,W)\otimes\Gamma^e\Hom(V,W).\]

In particular, for positive integers $n,d$ and a composition $\lambda=(\lambda_1,\lambda_2,\dots,\lambda_n)$ of $d$ in $n$ parts, 
i.e.\ an $n$-tuple of non negative integers such that $\sum_i \lambda_i = d$, we can form representable functors 
${\Gamma^{\lambda_1,k}\in\Rep\Gamma^{\lambda_1}_k},\dots,{\Gamma^{\lambda_n,k}\in\Rep\Gamma^{\lambda_n}_k}$
and  take their tensor product to obtain a functor in $\Rep\Gamma^d_k$
\begin{align*} \Gamma^\lambda\coloneqq\Gamma^{\lambda_1}\boxtimes\cdots\boxtimes\Gamma^{\lambda_n}. \end{align*}
In the same way define
\begin{align*} S^\lambda&\coloneqq S^{\lambda_1}\boxtimes\cdots\boxtimes S^{\lambda_n} \\
\Lambda^\lambda&\coloneqq\Lambda^{\lambda_1}\boxtimes\cdots\boxtimes\Lambda^{\lambda_n}. \end{align*}

\subsection*{Representations of Schur algebras}
For $n,d$ positive integers, the Schur algebra can be defined as 
\[S_k(n,d)=\End_{\frS_d}((k^n)^{\otimes d})=\End_{\Gamma^d_k}(\Gamma^{d,k^n})^{\op}.\]
If $n\geq d$ there is an equivalence of categories (\cite[Theorem 3.2]{FS1997} and \cite[Theorem 2.10]{Kr2013})
\begin{align}\label{eqschuralgebra}
\Rep\Gamma^d_k\xrightarrow{\cong}\Mod\End_{\Gamma^d_k}(\Gamma^{d,k^n})\cong S_k(n,d)\Mod,\end{align}
 given by evaluating at $k^n$, i.e.\ a strict polynomial functor $X$ is mapped to $X(k^n)$.

\subsection*{The internal tensor product of strict polynomial functors}
For $V, W$ in $\P_k$ denote by $V\otimes_k W$ the usual tensor product of $k$-modules. This induces a tensor product on 
$\Gamma^d\P_k$, the category of divided powers.
It coincides on objects with the one for $\P_k$ and on morphisms it is given via the following composite:
\begin{align*}\Gamma^d\Hom(V,V')\times\Gamma^d\Hom(W,W')\rightarrow\Gamma^d(\Hom(V,V')\otimes\Hom(W,W'))\\
\xrightarrow{\sim}\Gamma^d\Hom(V\otimes W,V'\otimes W').
\end{align*}

By Day convolution, this in turn yields an internal tensor product on $\Rep\Gamma^d_k$, such
that the Yoneda functor is closed strong monoidal. It is given for representable functors $\Gamma^{d,V}$ and
$\Gamma^{d,W}$ in $\Rep\Gamma^d_k$ by
\[\Gamma^{d,V}\otimes_{\Gamma^d_k}\Gamma^{d,W}:=\Gamma^{d,V\otimes W}.\]

For arbitrary objects it is given by taking colimits, see \cite[Proposition 2.4]{Kr2013} for more details.

The tensor unit is given by \[\I_{\Gamma^d_k}:=\Gamma^{d,k}\cong\Gamma^{(d)}.\]

In the same way,  $\Rep\Gamma^d_k$ is equipped with an internal hom, defined on representable objects by
\[\HOM_{\Gamma^d_k}(\Gamma^{d,V},\Gamma^{d,W}):=\Gamma^{d,\Hom(V,W)}.\]

The internal hom is indeed an adjoint to the internal tensor product, i.e.\ we have
a natural isomorphism (cf.\ \cite[Proposition 2.4]{Kr2013})
\[\Hom_{\Gamma^d_k}(X\otimes_{\Gamma^d_k}Y,Z)\cong\Hom_{\Gamma^d_k}(X,\HOM_{\Gamma^d_k}(Y,Z)).\]

We will omit the indices and write $-\otimes-$ and $\HOM(-,-)$ whenever it is clear which category is considered.

\subsection*{Dualities}
The category of strict polynomial functors admits two kinds of dual, one corresponding to the transpose duality
for modules over the general linear group
and the other one using the internal hom structure of $\Rep\Gamma^d_k$.
\subsubsection*{The Kuhn dual}
 For $X\in\Rep\Gamma^d_k$ it is defined by $X^\circ(V)\coloneqq X(V^*)^*$ where $(-)^*=\Hom_k(-,k)$ denotes the usual dual in $\P_k$.
Taking the Kuhn dual is a contravariant exact functor, sending projective objects to injective objects and vice versa.
Symmetric powers are duals of divided powers, i.e.\ $(\Gamma^d)^\circ=S^d$ and more generally $(\Gamma^\lambda)^\circ=S^\lambda$.
Exterior powers are self-dual, i.e.\ $(\Lambda^\lambda)^\circ=\Lambda^\lambda$.

\subsubsection*{The monoidal dual}
It is defined for $X\in\Rep\Gamma^d_k$ by $X^\vee\coloneqq\HOM_{\Gamma^d_k}(X,\Gamma^d)$.
 This functor is left exact, but in general not right exact.

\begin{lemma}\cite[Lemma 2.7 and Lemma 2.8]{Kr2013}\label{lem:2.7-2.8}
For all $X,Y\in\Rep\Gamma^d_k$ we have a natural isomorphism
\[\HOM_{\Gamma^d\P_k}(X,Y^\circ)\cong\HOM_{\Gamma^d\P_k}(Y,X^\circ).\]
If $X$ is finitely presented we have  natural isomorphisms
\begin{align*}
X\otimes_{\Gamma^d_k}Y^\circ&\cong\HOM_{\Gamma^d_k}(X,Y)^\circ\\
(X\otimes_{\Gamma^d_k}Y)^\circ&\cong\HOM_{\Gamma^d_k}(X,Y^\circ).
\end{align*}
\end{lemma}

We collect some important calculations:

\begin{align}
X\otimes\Gamma^d&\cong X\\
S^d\otimes S^d&\cong S^d \label{S^dotimesS^d} \\
(\Gamma^d)^\vee=\HOM(\Gamma^d,\Gamma^d)&\cong(\Gamma^d\otimes S^d)^\circ\cong\Gamma^d\\
\Gamma^\lambda\otimes S^d&\cong S^\lambda\label{GammaxS}\\
\HOM(S^d,S^\lambda)&\cong S^\lambda\label{S,S}
\end{align}

\section{Representations of the symmetric group and the Schur functor}
Recall that $\frS_d$ is the symmetric group permuting $d$ elements. The representations of $\frS_d$, i.e.\ (left) modules
over its group algebra will be denoted by $k\frS_d\Mod$. Define $k\frS_d\mod$ to be the subcategory of modules
that are finitely generated projective over $k$.

\subsection*{Partitions}
We will denote by $\Lambda(n,d)\coloneqq\{\lambda=(\lambda_1,\dots,\lambda_n)\mid\sum\lambda_i=d\}$ the set of all
compositions of $d$ into $n$ parts. Those compositions that are weakly decreasing, i.e.\
$\lambda_1\geq\lambda_2\geq\dots\geq\lambda_n\geq 0$, are called \emph{partitions} and denoted by $\Lambda^+(n,d)$.
The subset of \emph{$p$-restricted partitions}, i.e.\  $\lambda\in\Lambda^+(n,d)$ with $\lambda_i-\lambda_{i+1}<p$, are
denoted by  $\Lambda^+_p(n,d)$.
A sequence $(i_1\dots i_d)$ \emph{belongs to $\lambda$}, denoted as $(i_1\dots i_d)\in\lambda$, if 
$(i_1\dots i_d)$  has $\lambda_{l}$ entries equal to $l$.

\subsection*{Permutation modules} 
Fix a basis $e_1,\dots,e_n$ of $k^n$ and consider the $d$-fold tensor product $(k^n)^{\otimes d}$. It
becomes a left $k\frS_d$-module by defining the module action via
  \begin{align*}
\sigma(v_1 \otimes\dots\otimes v_d) := v_{\sigma^{-1}(1)}\otimes \dots \otimes v_{\sigma^{-1}(d)}.
  \end{align*}
for  $\sigma \in \frS_d$ and $v_1 \otimes \dots \otimes v_d \in (k^n)^{\otimes d}$. It decomposes into 
a direct sum of \emph{transitive permutation modules}
 \begin{align}\label{dec_perm}
(k^n)^{\otimes d}=\bigoplus_{\lambda \in \Lambda(n,d)} M^{\lambda},
\end{align} where $M^{\lambda}$ is the $k$-span of the set $\{e_{i_1}\otimes\dots\otimes e_{i_d} \ | \ (i_1\dots i_d)\in\lambda \}$.

\subsection*{The internal tensor product of representations of symmetric groups}
The Hopf algebra structure of the group algebra $k\frS_d$ endows $k\frS_d\Mod$ with an internal tensor product, the so-called Kronecker product.
For $N,~{N'}\in k\frS_d\Mod$ it is given by taking the usual tensor product over $k$, denoted by $N\otimes_{k}{N'}$,
together with the following diagonal action of $\sigma\in\frS_d$:
\begin{align*}
 \sigma\cdot(n\otimes n')= \sigma n\otimes\sigma n'.
\end{align*}
The tensor unit is given by $M^{(d)}\cong k$, the trivial $k\frS_d$-module.

From the antipode which is defined by $S(\sigma)=\sigma^{-1}$ for $\sigma\in\frS_d$, we also get an internal hom, 
denoted by $\HOM_k(N,{N'})$. It is given by
taking $k$-linear morphisms $\Hom_k(N,{N'})$ together with
the following action for $\sigma\in\frS_d$,  $f\in\Hom_k(N,{N'})$ and $n\in N$:
\[ \sigma\cdot f(n)=\sigma f(S(\sigma) n)=\sigma f(\sigma^{-1} n).\]

\subsection*{Duality}
By setting ${N'}\coloneqq \I_{k\frS_d}=k$, the trivial $k\frS_d$-module, the internal hom provides a dual $N^*\coloneqq \HOM_k(N,k)$. The action becomes
\[ \sigma\cdot f(n)=\sigma f(\sigma^{-1}n)=f(\sigma^{-1} n).\]

\subsection*{The Schur functor}
The Schur functor was originally defined from representations of the Schur algebra to representation of the symmetric group.
Via the equi\-valence (\ref{eqschuralgebra}) this translates to a functor from the category of strict polynomial functors 
to the representations of the symmetric group.

Let  $\omega=(1,\dots,1)\in\Lambda(d,d)$ be the partition with $d$ entries equal to $1$.
  There is (cf.\ \cite[Section 4]{Kr2014}) an algebra isomorphism
\[\End_{\Gamma^d_k}(\Gamma^\omega)\cong  k\frS_d^\op,\]
which identifies the module categories   $k\frS_d\Mod$ and $\Mod\End_{\Gamma^d_k}(\Gamma^\omega)$. In the following we
will write $\End(\Gamma^\omega)$ instead of $\End_{\Gamma^d_k}(\Gamma^\omega)$.

\begin{definition} The \emph{Schur functor}, denoted by $\F$, is defined as
\[\F:=\Hom_{\Gamma^d_k}(\Gamma^\omega,-)\colon\Rep\Gamma^d_k\to\Mod\End(\Gamma^\omega)\cong k\frS_d\Mod.\]
\end{definition}

\subsection*{An equivalence of categories}
Let $\Gamma=\{\Gamma^\lambda\}_{\lambda\in\Lambda(n,d)}$, $M=\{M^\lambda\}_{\lambda\in\Lambda(n,d)}$ 
and $S=\{S^\lambda\}_{\lambda\in\Lambda(n,d)}$. Denote by $\add\Gamma$  resp.\ $\add S$ the full subcategory of $\Rep\Gamma^d_k$ whose 
objects are direct summands  of finite direct sums of $\Gamma^\lambda$ resp.\ $S^\lambda$. Define $\add M$ similarly as a subcategory of $k\frS_d\Mod$.
In \cite[Lemma 4.3]{AR2015} it is shown that the functor $\F=\Hom_{\Gamma^d_k}(\Gamma^\omega,-)$ 
induces an equivalence of categories between $\add\Gamma$ and $\add M$. By taking duals we get that $\F$ also induces
an equivalence of categories between $\add S$ and $\add M$. 

Note that $\add\Gamma$ is the subcategory of 
consisting of all finitely generated projective objects and $\add S$ is the subcategory of finitely generated injective objects.

\section{The left adjoint of the Schur functor}
In this section we present the connection between the left adjoint of the Schur functor and the monoidal structure
of strict polynomial functors. In addition we show that the left adjoint induces an equivalence between some subcategories of 
$\frS_d\Mod$ resp.\ of $\Rep\Gamma^d_k$.

Let $N\in\Mod\End(\Gamma^\omega)$ and $X\in\Rep\Gamma^d_k$. By the usual tensor-hom adjunction we get the following isomorphism
\begin{align*}\Hom_{\Gamma^d_k}(N\otimes_{\End(\Gamma^\omega)}\Gamma^\omega,X)&\cong\Hom_{\End(\Gamma^\omega)}(N,\Hom_{\Gamma^d_k}(\Gamma^\omega,X))\\
&=\Hom_{\End(\Gamma^\omega)}(N,\F(X)).
\end{align*}
Thus, $\F$ has a left adjoint, namely
\begin{align*}
\G_{\otimes}\colon\Mod\End(\Gamma^\omega)&\to\Rep\Gamma^d_k\\
N&\mapsto N\otimes_{\End(\Gamma^\omega)}\Gamma^\omega,
\intertext{in terms of modules for the symmetric group algebra this reads }
\quad \G_{\otimes}\colon k\frS_d\Mod&\rightarrow \Rep\Gamma^d_k\\
 N&\mapsto (-)^{\otimes d}\otimes_{k\frS_d}N.
\end{align*}
We will denote the unit by $\eta_\otimes\colon\id_{\End(\Gamma^\omega)}\to\F\G_{\otimes}$ and the counit by 
$\varepsilon_\otimes\colon\G_{\otimes}\F\to\id_{\Rep\Gamma^d_k}$ and omit indices where possible. Note that $\G_{\otimes}$ is fully faithful,
hence the unit $\eta_\otimes$ is an isomorphism, i.e.\ $\F\G_{\otimes}(X)\cong X$ for all $X\in\Rep\Gamma^d_k$. 

\medskip

We are now interested in the composition $\G_\otimes\F$.

\medskip

\begin{proposition}\label{Prop:GotimesF}
 There is a natural isomorphism
 \[\G_{\otimes}\F(X)\cong X\]
 for all $X\in\add S$.
\end{proposition}
\begin{proof}
 Let $V\in\Gamma^d\P_k$ and $X\in\add S$. Using Lemma \ref{lem:2.7-2.8}, the Yoneda isomorphism (\ref{Yoneda}), and the equivalence of $\add S$ and $\add M$ we get
 the following sequence of isomorphisms,
 \begin{align*}(X)^\circ(V)&{\cong}\Hom_{\Gamma^d_k}(\Gamma^{d,V},(X)^\circ)\\
  &{\cong}\Hom_{\Gamma^d_k}(X,(\Gamma^{d,V})^\circ)\\
  &{\cong}\Hom_{k\frS_d}(\F(X),\F((\Gamma^{d,V})^\circ))\\
    &{\cong}\Hom_{\Gamma^d_k}(\G_{\otimes}\F(X),(\Gamma^{d,V})^\circ)\\
    &{\cong}\Hom_{\Gamma^d_k}(\Gamma^{d,V},(\G_{\otimes}\F(X))^\circ)\\
    &{\cong}(\G_{\otimes}\F(X))^\circ(V) 
\end{align*}
and thus $\G_{\otimes}\F(X)\cong X$.
\end{proof}

\begin{corollary}\label{Cor:eqGotimes}
The functor $\G_\otimes$ restricted to $\add M$ is an inverse of $\F|_{\add S}$, i.e.\ we have the following equivalences of categories
 \[\pushQED{\qed}\xymatrix{
 \add S\ar@/_/[r]_{\F}&\add M\ar@/_/[l]_{\G_{\otimes}}
 }\qedhere
\popQED\]
\end{corollary}

If we do not restrict to the subcategory $\add S$, the composition $\G_\otimes\F$ is not isomorphic to the identity. 
Though we have the following result:

\begin{theorem}\label{Th:GotimesF}
There is a natural isomorphism
\[\G_{\otimes}\F(X)\cong S^d\otimes_{\Gamma^d_k} X.\]
\end{theorem}

\begin{proof} 
By \cite[Theorem 4.4]{AR2015} the functor $\F$ is monoidal and thus there is a natural isomorphism \[\Phi_{X,Y}\colon\F(X)\otimes_k\F(Y)\to\F(X\otimes_{\Gamma^d_k}Y).\]
Using this isomorphism and by adjunction we get a sequence of isomorphisms
 \begin{align*}\Hom_{k\frS_d}(\F(X)\otimes_k N,\F(X)\otimes_k N)&\cong\Hom_{k\frS_d}(\F(X)\otimes_kN,\F(X)\otimes_k\F\G_{\otimes}(N))\\
  &\cong\Hom_{k\frS_d}(\F(X)\otimes_kN,\F(X\otimes_{\Gamma^d_k}\G_{\otimes}(N)))\\
  &\cong\Hom_{\Gamma^d_k}(\G_{\otimes}(\F(X)\otimes_kN),X\otimes_{\Gamma^d_k}\G_{\otimes}(N)).
 \end{align*}
Thus, the identity on $\F(X)\otimes_k N$ yields a map $\vartheta_{X,N}$ which is given by 
\[\vartheta_{X,N}\coloneqq(\varepsilon_{\otimes})\circ\G_{\otimes}(\Phi\circ(\id\otimes_k \eta_{\otimes}))
\colon\G_{\otimes}(\F(X)\otimes_kN)\to X\otimes_{\Gamma^d_k}\G_{\otimes}(N).\]
 By  setting $N\coloneqq\I$, the trivial module, 
we get a map
 \[\vartheta_{X,\I}\colon\G_{\otimes}\F(X)\to X\otimes_{\Gamma^d_k}\G_{\otimes}(\I).\]
 We will show that it is an isomorphism.
Since $\G_{\otimes}\F(-)$ and $-\otimes_{\Gamma^d_k}\G_{\otimes}(\I)$ are right exact functors it is enough to show that $\vartheta_{X,\I}$ is
an isomorphism for $X$ projective. 
Thus, let $X=\Gamma^\lambda$. Since $\F(S^d)\cong \I$,  we know by  Proposition \ref{Prop:GotimesF} that $\G_{\otimes}(\I)\cong S^d$. It follows that
$\Gamma^\lambda\otimes_{\Gamma^d_k}\G_{\otimes}(\I)\cong S^\lambda$ by (\ref{GammaxS}) and that
\[(\varepsilon_{\otimes})_{\Gamma^\lambda\otimes_{\Gamma^d_k}\G_{\otimes}(\I)}\colon\G_{\otimes}\F(\Gamma^\lambda\otimes_{\Gamma^d_k}\G_{\otimes}(\I))
\to \Gamma^\lambda\otimes_{\Gamma^d_k}\G_{\otimes}(\I)\]
is an isomorphism by Corollary \ref{Cor:eqGotimes}. Both maps $\Phi$ and $\eta_{\otimes}$ are isomorphisms and thus 
$\G_{\otimes}(\Phi\circ(\id\otimes_k \eta_{\otimes}))$ is an isomorphism. It follows that 
$\vartheta_{\Gamma^\lambda,\I}\colon\G_{\otimes}\F(\Gamma^\lambda)\to\Gamma^\lambda\otimes_{\Gamma^d_k}\G_{\otimes}(\I)$ is an isomorphism.
Identifying $\G_{\otimes}(\I)$ with $S^d$ we get the desired isomorphism
\[\vartheta_{X,\I}\colon\G_{\otimes}\F(X)\xrightarrow{\cong} X\otimes_{\Gamma^d_k}S^d\cong  S^d\otimes_{\Gamma^d_k} X.\qedhere\]
\end{proof}

Recall from (\ref{S^dotimesS^d}) that $S^d\cong S^d\otimes S^d$. Thus, by using the fact that $\F$ preserves the monoidal structure, we get the following
\begin{corollary}The functor $G_{\otimes}$ is compatible with the tensor product, i.e.\ 
  \[\pushQED{\qed}\G_{\otimes}(N\otimes_k {N'})\cong\G_{\otimes}(N)\otimes_{\Gamma^d_k}\G_{\otimes}({N'}).\qedhere
\popQED\]
\end{corollary}
 However, note that the tensor unit $\I_{k\frS_d}$ is mapped under $\G_{\otimes}$ to $S^d$
which is not the tensor unit in $\Rep\Gamma^d_k$.
Using Lemma \ref{lem:2.7-2.8}, we get the following description of 
the Schur functor composed with its left adjoint:
\begin{corollary}
We can express the endofunctor  $\G_{\otimes}\F$ by duals, namely
 \[\pushQED{\qed}
\G_{\otimes}\F(X)\cong S^d\otimes_{\Gamma^d_k} X\cong\HOM(X,\Gamma^d)^\circ=(X^\vee)^\circ\qedhere
\popQED\]
\end{corollary}

\section{The right adjoint of the Schur functor}
This section provides analogous results to those in the preceeding section, 
now for the right adjoint of the Schur functor. In particular, we will
see how the right adjoint can be expressed  in terms of the monoidal structure
of strict polynomial functors.

Let $V\in\Gamma^d\P_k$, $X\in\Rep\Gamma^d_k$ and $N\in\Mod\End(\Gamma^\omega)$. We write $\End(\Gamma^{d,V})$
for $\End_{\Gamma^d_k}(\Gamma^{d,V})$ and consider $\Hom_{\Gamma^d_k}(\Gamma^{d,V},X)$
as a right $\End(\Gamma^{d,V})$-module and  $\Hom_{\Gamma^d_k}(\Gamma^\omega,\Gamma^{d,V})$ as an 
$\End(\Gamma^{d,V})$-$\End(\Gamma^\omega)$-bimodule.
By the usual tensor-hom adjunction we then get the 
following isomorphism
\begin{align*}\Hom_{\End(\Gamma^{d,V})}(\Hom_{\Gamma^d_k}(\Gamma^{d,V},X),\Hom_{\End(\Gamma^\omega)}(\Hom_{\Gamma^d_k}(\Gamma^\omega,\Gamma^{d,V}),N))\\
\cong\Hom_{\End(\Gamma^\omega)}(\Hom_{\Gamma^d_k}(\Gamma^{d,V},X)\otimes_{\End(\Gamma^{d,V})}\Hom_{\Gamma^d_k}(\Gamma^\omega,\Gamma^{d,V}),N).\end{align*}
On the other hand, since $\Hom_{\Gamma^d_k}(\Gamma^{d,V},\Gamma^\omega)$ is finitely generated projective over $\End(\Gamma^{d,V})$, we also have
\begin{align*}
\Hom_{\End(\Gamma^{d,V})}(\Gamma^\omega(V),X(V))&
\cong\Hom_{\End(\Gamma^{d,V})}(\Hom_{\Gamma^d_k}(\Gamma^{d,V},\Gamma^\omega),\Hom_{\Gamma^d_k}(\Gamma^{d,V},X))\\
&\cong \Hom_{\Gamma^d_k}(\Gamma^{d,V},X) \otimes_{\End(\Gamma^{d,V})}  \Hom_{\Gamma^d_k}(\Gamma^\omega,\Gamma^{d,V})
\end{align*} and thus 
\begin{align*}\Hom_{\End(\Gamma^{d,V})}(\Hom_{\Gamma^d_k}(\Gamma^{d,V},X),\Hom_{\End(\Gamma^\omega)}(\Hom_{\Gamma^d_k}(\Gamma^\omega,\Gamma^{d,V}),N))\\
\cong\Hom_{\End(\Gamma^\omega)}(\Hom_{\End(\Gamma^{d,V})}(\Gamma^\omega(V),X(V)),N).\end{align*}
Since $\Mod\End(\Gamma^{d,V})\cong\Rep\Gamma^d_k$ for $V\coloneqq k^n$ with $n\geq d$ by (\ref{eqschuralgebra}) 
and $X\cong\Hom_{\Gamma^d_k}(\Gamma^{d,-},X) $ this isomorphism finally becomes
\begin{align*}\Hom_{\Gamma^d_k}(X,\Hom_{\End(\Gamma^\omega)}(\Hom_{\Gamma^d_k}(\Gamma^\omega,\Gamma^{d,-}),N))\\
\cong\Hom_{\End(\Gamma^\omega)}(\Hom_{\Gamma^d_k}(\Gamma^\omega,X),N).\end{align*}
Thus, 
 $\F=\Hom_{\Gamma^d_k}(\Gamma^\omega,-)$ has a right adjoint, namely
\begin{align*}
 \G_{\Hom}\colon \Mod\End(\Gamma^\omega)&\rightarrow \Rep\Gamma^d_k\\
 N&\mapsto\Hom_{\End(\Gamma^\omega)}(\Hom_{\Gamma^d_k}(\Gamma^\omega,\Gamma^{d,-}),N)\\
 \intertext{in terms of modules for the symmetric group algebra this reads }
 \G_{\Hom}\colon k\frS_d\Mod&\rightarrow \Rep\Gamma^d_k\\
 N&\mapsto\Hom_{k\frS_d}(\Hom_{k\frS_d}((-)^{\otimes d},k\frS_d),N).\\
\end{align*}
We will denote the unit by $\eta_{\Hom}\colon\id_{\End(\Gamma^\omega)}\to\G_{\Hom}\F$ and the counit by \linebreak
$\varepsilon_{\Hom}\colon\F\G_{\Hom}\to\id_{\Rep\Gamma^d_k}$. 
Note that  $\G_{\Hom}$ is fully faithful,
hence the counit $\varepsilon_{\Hom}$ is an isomorphism,  i.e.\ $\F\G_{\Hom}(X)\cong X$ for all $X\in\Rep\Gamma^d_k$.

\medskip

Again, we are interested in the composition $\G_{\Hom}\F$. We have the following result, dual to Proposition \ref{Prop:GotimesF}:

\medskip

\begin{proposition}\label{Prop:GHomF}
 There is a natural isomorphism
 \[\G_{\Hom}\F(X)\cong X\]
for all  $X\in\add \Gamma$.
\end{proposition}
\begin{proof}
 Let $V\in\Gamma^d\P_k$ and $X\in\add \Gamma$. Due to the the Yoneda isomorphism (\ref{Yoneda}) and the equivalence 
 of $\add\Gamma$ and $\add M$ we have the following sequence of isomorphisms 
 \begin{align*}X(V)&{\cong}\Hom_{\Gamma^d_k}(\Gamma^{d,V},X)\\
  &{\cong}\Hom_{k\frS_d}(\F(\Gamma^{d,V}),\F(X))\\
   &{\cong}\Hom_{\Gamma^d_k}(\Gamma^{d,V},\G_{\Hom}\F(X))\\
   &{\cong}\G_{\Hom}\F(X)(V)
\end{align*}
and thus $\G_{\Hom}\F(X)\cong X$.
\end{proof}

\begin{corollary}\label{Cor:eqGHom}
The functor $\G_{\Hom}$ restricted to $\add M$ is an inverse of $\F|_{\add \Gamma}$, i.e.\ we have the following equivalences of categories
\[\pushQED{\qed}\xymatrix{
 \add \Gamma\ar@/^/[r]^{\F}&\add M\ar@/^/[l]^{\G_{\Hom}}
 }\qedhere
\popQED\]
\end{corollary}

\begin{remark}
Suppose that $k$ is a field of characteristic $\geq 5$. In \cite[Theorem 3.8.1.]{HN2004} it is shown that on $\Filt(\Delta)$,
the full subcategory of Weyl filtered modules, $\G_{\Hom}$ is an inverse to $\F$. The subcategory  $\Filt(\Delta)$ contains
the subcategory $\add\Gamma$, so in the case of a field of characteristic $\geq 5$, Corollary \ref{Cor:eqGHom} follows also
from \cite{HN2004}. But note that Corollary \ref{Cor:eqGHom} is independent of any assumption on the
commutative ring $k$.
\end{remark}

If we do not restrict to the subcategory $\add\Gamma $, the composition $\G_{\Hom}\F$ is not isomorphic to the identity. 
Though we have the following result, dual to Theorem \ref{Th:GotimesF}:

\begin{theorem}\label{Th:GHomF}
There is a natural isomorphism
\[\G_{\Hom}\F(X)\cong \HOM(S^d,X).\]
\end{theorem}

\begin{proof} 
By using the fact that $\F$ is monoidal (\cite[Theorem 4.4]{AR2015}) and some
additional calculations, one can show that
 $\F$ is also a closed functor, i.e.\ there is a natural isomorphism 
\[\Psi_{X,Y}\colon\HOM_{k\frS_d}(\F(X),\F(Y))\to\F(\HOM_{\Rep\Gamma^d_k}(X,Y)).\]
Using this isomorphism and by adjunction we get a sequence of isomorphisms
 \begin{align*}&\Hom_{k\frS_d}(\HOM(\F(X^\circ),N),\HOM(\F(X^\circ),N))\\
 \cong&\Hom_{k\frS_d}(\HOM(\F(X^\circ),N),\HOM(N^*,\F(X^\circ)^*))\\
\cong&\Hom_{k\frS_d}(\HOM(\F(X^\circ),\F\G_{\Hom}(N)),\HOM(N^*,\F(X)))\\
\cong&\Hom_{k\frS_d}(\F(\HOM(X^\circ,\G_{\Hom}(N))),\HOM(N^*,\F(X)))\\
\cong&\Hom_{\Gamma^d_k}(\HOM(X^\circ,\G_{\Hom}(N)),\G_{\Hom}(\HOM(N^*,\F(X))))\\
\cong&\Hom_{\Gamma^d_k}(\HOM(\G_{\Hom}(N)^\circ,X),\G_{\Hom}(\HOM(N^*,\F(X)))).
 \end{align*}
Thus, the identity on $\HOM(\F(X^\circ),N)$ yields a map
\[\kappa_{N,X}
\colon\HOM(\G_{\Hom}(N)^\circ,X)\to\G_{\Hom}(\HOM(N^*,\F(X^\circ))).\]
 By setting $N\coloneqq\I$, the trivial module, 
we get a map
 \[\kappa_{\I,X}\colon\HOM(\G_{\Hom}(\I)^\circ,X)\to\G_{\Hom}(\HOM(\I,\F(X))).\]
Similarly to the case of $\G_{\otimes}$ this is an isomorphism. This time, we use the fact that
since $\HOM(\G_{\Hom}(\I),-)$ and $\G_{\Hom}(\HOM(\I,\F(-)))$ are left exact functors it is enough to show that $\kappa_{\I,X}$ is
an isomorphism for $X=S^\lambda$ injective. 
But $\F(\Gamma^d)\cong \I$, thus  we know by  Proposition \ref{Prop:GHomF} that $\G_{\Hom}(\I)\cong \Gamma^d$. It follows that
\[\HOM(\G_{\Hom}(\I)^\circ,S^\lambda)\cong \Gamma^\lambda\] by (\ref{S,S}) and hence
\[(\eta_{\Hom})_{\HOM(\G_{\Hom}(\I)^\circ,S^\lambda)}\colon\HOM(\G_{\Hom}(\I)^\circ,S^\lambda)
\to\G_{\Hom}\F(\HOM(\G_{\Hom}(\I)^\circ,S^\lambda))\]
is an isomorphism by Corollary \ref{Cor:eqGHom}.
Similarly as before $\kappa_{\I,X}$ is the composition of this isomorphism and further isomorphisms, hence is
itself and isomorphism.
Identifying $\G_{\Hom}(\I)^\circ$ with $S^d$ and $\HOM(\I,\F(X))$
with $\F(X)$
we finally get the desired isomorphism
\[\kappa_{\I,X}\colon\HOM(S^d,X)\to\G_{\Hom}\F(X).\qedhere\]
\end{proof}

\begin{corollary} The functor $\G_{\Hom}$ preserves the internal hom up to duality, i.e.\
 \[\G_{\Hom}\F(\HOM(X^\circ,Y))\cong\HOM(\G_{\Hom}\F(X)^\circ,\G_{\Hom}\F(Y))\] and
  \[\pushQED{\qed}\G_{\Hom}\HOM(N^*,{N'})\cong\HOM(\G_{\Hom}(N)^\circ,\G_{\Hom}({N'})).\qedhere
\popQED\]
\end{corollary}

\begin{remark}
 In general, $\G_{\Hom}$ does not preserve the internal tensor product, e.g.\ 
 defining $\sgn$ to be the sign-representation in $k\frS_d\Mod$ 
we get $\G_{\Hom}(\sgn)\cong \Lambda^d$ if $2$ is invertible in $k$, but  $\G_{\Hom}(\I)\cong\Gamma^d$ and thus
\begin{align*}
\G_{\Hom}(\sgn\otimes_k\sgn)
=\G_{\Hom}(\I)\cong\Gamma^d\neq S^d&\cong\Lambda^d\otimes_{\Gamma^d_k}\Lambda^d\\
&\cong\G_{\Hom}(\sgn)\otimes_{\Gamma^d_k}\G_{\Hom}(\sgn).
              \end{align*}
\end{remark}

\begin{corollary}Let $X\in\rep\Gamma^d_k$, i.e.\ $X^{\circ\circ}\cong X$.
We can express the endofunctor  $\G_{\Hom}\F$ by duals, namely
\[\pushQED{\qed}\G_{\Hom}\F(X)\cong \HOM(S^d,X)\cong\HOM(X^\circ,\Gamma^d)=(X^\circ)^\vee.\qedhere
\popQED\]
\end{corollary}

\section{Comparing both adjoints}
The results in the previous two sections allow us to relate the left and the right adjoint. In the
case of $k$ a field of characteristic $p$, this has already been done in a more general setting by N.~Kuhn in 
\cite[Theorem~6.10, Lemma~6.11]{Ku2002}.

In our setting, $k$ is still an (arbitrary) commutative ring and we obtain

\begin{proposition}
The left and the right adjoints of the Schur functor are related by taking duals, namely
 \begin{align*}
(\G_{\otimes}\circ\F(X))^\circ\cong \G_{\Hom}\circ\F(X^\circ)
 \end{align*} 
 and 
 \[\G_{\otimes}(N)^\circ\cong\G_{\Hom}(N^*)\]
 for all $X\in\Rep\Gamma^d_k$ and $N\in k\frS_d\Mod$.
\end{proposition}

\begin{proof}
 Using Theorem~\ref{Th:GHomF} and Theorem~\ref{Th:GotimesF} we get
\[ (\G_{\otimes}\circ\F(X))^\circ\cong (S^d\otimes X)^\circ\cong\HOM(S^d,X^\circ)\cong \G_{\Hom}\circ\F(X^\circ).\]
 By setting $N\coloneqq\F(X)$ and using the fact that $\F(X^\circ)\cong\F(X)^*$ we get the second isomorphism.
\end{proof}

 We have the following commutative diagram

\[\xymatrix{
 \rep\Gamma^d_k\ar@/_/[r]_{\F}\ar[d]_{(-)^\circ}&k\frS_d\mod\ar@/_/[l]_{\G_{\otimes}}\ar[d]^{(-)^*}\\
 (\rep\Gamma^d_k)^\op\ar@/^/[r]^{\F}&(k\frS_d\mod)^\op\ar@/^/[l]^{\G_{\Hom}}
}\]  
where the vertical arrows are equivalences of categories. The horizontal arrows become
equivalences when restricted to the following subcategories 
 
 \[\xymatrix{
 \add S\ar@/_/[r]_{\F}\ar[d]_{(-)^\circ}&\add M\ar@/_/[l]_{\G_{\otimes}}\ar[d]^{(-)^*}\\
 (\add \Gamma)^\op\ar@/^/[r]^{\F}&(\add M)^\op\ar@/^/[l]^{\G_{\Hom}}
}\]

\section{The tensor product of simple functors}
If $k$ is a field, the isomorphism classes of simple functors in $\Rep\Gamma^d_k$ are indexed by partitions $\lambda\in\Lambda^+(n,d)$.
Simple functors are self-dual, i.e.\ $L_\lambda^\circ\cong L_\lambda$, 
see e.g.\ \cite[Proposition 4.11]{Kr2014}. 

In \cite[Theorem~7.11]{Ku2002} a generalized Steinberg Tensor Product Theorem is proved that states that
simple functors are given by the external tensor product of twisted simple functors.
In our setting, this has been formulated also by Touz\'e:
\begin{theorem}[{\cite[Theorem 4.8]{Tou2015}}]
 Let $k$ be a field of charachteristic $p$. Let $\lambda^0,\dots,\lambda^r$ be $p$-restricted partitions, and let $\lambda=\sum_{i=0}^rp^i\lambda^i$. There is an isomorphism:
 \[L_\lambda\cong L_{\lambda^0}\boxtimes L_{\lambda^1}^{(1)}\boxtimes\dots\boxtimes L_{\lambda^r}^{(r)},\]
where $(-)^{(i)}$ denotes the $i$-th Frobenius twist.
\end{theorem}
Using this decomposition, Touz\'e shows that for calculating the internal tensor product of two simple functors, it is
enough to consider $p$-restricted partitions. Namely one has for $\lambda =\sum_{i=0}^rp^i\lambda^i$
and $\mu=\sum_{i=0}^sp^i\mu^i$
\[L_\lambda\otimes L_\mu\cong
\begin{cases}
(L_{\lambda^0}\otimes L_{\mu^0})\boxtimes(L_{\lambda^1}\otimes L_{\mu^1})^{(1)}\boxtimes\dots
\boxtimes(L_{\lambda^r}\otimes L_{\mu^r})^{(r)},& r=s, |\lambda^i|= |\mu^i| \\
0&\text{ otherwise.}
\end{cases}\](see \cite[Theorem 6.2]{Tou2015}). 
Unfortunately, the tensor product of two simple functors
is almost never simple, as we will see in the next theorem.

Throughout this section, let $p$ be the characteristic of $k$.

\subsection*{Mullineux map and truncated symmetric powers}
Denote by $m$ the Mulli\-neux map $m\colon\Lambda^+_p(n,d)\to\Lambda^+_p(n,d)$ that relates
simple $k\frS_d$-modules (see e.g.\ \cite[Chapter 4.2]{M1993}). Define $Q^d$ to be the truncated symmetric powers, i.e.\ the top of $S^d$. 
We have the following connection between tensor products of simple functors
and the Schur functor and its left adjoint:

\begin{lemma}\label{Lem:con_simple_adj}
Let  $\mu$ be a $p$-restricted partition, i.e.\ $\mu\in\Lambda^+_p(n,d)$. Then
 \[Q^d\otimes L_\mu\cong\G_{\otimes}\F(L_{\mu}), \quad \Lambda^d\otimes L_\mu\cong\G_{\otimes}\F(L_{m(\mu)})\]
\end{lemma}

\begin{proof}
  From \cite[Corollary 6.9, 6.10]{Tou2015} we know that
\[Q^d\otimes L_\mu\cong Q^d\otimes\Lambda^d\otimes\Lambda^d\otimes L_\mu\cong\Lambda^d\otimes\Lambda^d\otimes L_\mu\cong S^d \otimes L_\mu.\]
By Theorem \ref{Th:GotimesF} this is the same as applying the Schur functor and its left adjoint to $L_{\mu}$.

Again by \cite[Corollary 6.9]{Tou2015} we have $\Lambda^d\otimes L_\mu\cong Q^d\otimes L_{m(\mu)}$ and thus 
$\Lambda^d\otimes L_\mu\cong\G_{\otimes}\F(L_{m(\mu)})$
\end{proof}

The following lemma shows in which cases the right adjoint of the Schur functor sends simple modules to simple functors.
I am very grateful to Karin Erdmann who pointed out the connection between
the occurence of composition factors in quotients of projective covers and $\Ext$-vanishing of simple functors.

\begin{lemma}\label{Lem:projcovervsExt} Let $\mu\in\Lambda^+_p(n,d)$. Then
  $\G_{\otimes}\F(L_\mu)\cong L_\mu$ if and only if all $\nu$ with $\Ext^1(L_\mu,L_\nu)\neq 0$ are $p$-restricted.
\end{lemma}

\begin{proof}
We use \cite[3.2 Corollary]{DEN2004} that states that $\G_{\otimes}\F(L_\mu)$ is the largest quotient of the projective cover $P_\mu$
of $L_{m(\mu)}$ whose radical has only non $p$-restricted composition factors. This is simple if and only if 
the top of $\rad P_\mu$ has only $p$-restricted composition factors.

If we apply $\Hom(-,L_\nu)$ to the exact sequence
\[0\to\rad P_{\mu}\to P_\mu\to L_\mu\to 0,\]
we get
\begin{align*}0\to\Hom(L_\mu,L_\nu)\to \Hom(P_\mu,L_\nu)&\to \Hom(\rad P_{\mu},L_\nu)\\
&\to\Ext^1(L_\mu,L_\nu)\to \Ext^1(P_\mu,L_\nu)=0.
\end{align*}
Since $\Hom(L_\mu,L_\nu)\cong \Hom(P_\mu,L_\nu)$, we obtain $\Hom(\rad P_{\mu},L_\nu)\cong\Ext^1(L_\mu,L_\nu)$.
That means, $L_\nu$ is a composition factor of the top of $\rad P_\mu$ if and only if \linebreak $\Ext^1(L_\mu,L_\nu)\neq 0$. 
So, we get
\begin{align*}
\G_{\otimes}\F(L_{\mu}) \text{ is simple }\Leftrightarrow \text{ all }
\nu \text{ with } \Ext^1(L_\mu,L_\nu)\neq 0 \text{ are $p$-restricted.}
\end{align*}
Since $\F\G_{\otimes}\F(L_\mu)\cong\F(L_\mu)$ we know that if $\G_{\otimes}\F(L_{\mu}) $ is simple, it must be isomorphic to $L_\mu$.
\end{proof}

Finally, we get the following characterization of tensor products of simple functors corresponding to
$p$-restricted partitions that are again simple.
\begin{theorem}\label{Th:simples}
 Let $k$ be a field of odd characteristic and $\lambda,\mu\in\Lambda^+_p(n,d)$. The tensor
 product $L_\lambda\otimes L_\mu$ is simple if and only if, up to interchanging $\lambda$ and $\mu$,
 \begin{compactitem}
  \item[-] $L_\lambda\cong\Lambda^d$ and all $\nu$ with $\Ext^1(L_{m(\mu)},L_\nu)\neq 0$ are $p$-restricted, or
  \item[-]  $L_\lambda\cong Q^d$ and  all $\nu$ with $\Ext^1(L_\mu,L_\nu)\neq 0$ are $p$-restricted.
 \end{compactitem}
 In these cases $\Lambda^d\otimes L_\mu\cong L_{m(\mu)}$ and $Q^d\otimes L_\mu\cong L_{\mu}$.
\end{theorem}

\begin{proof}
First note that if $\dim\F(L_\lambda)\geq 2$ and $\dim\F(L_\mu)\geq 2$, then $L_\lambda\otimes L_\mu$ is not
simple (see \cite[Corollary 6.6]{Tou2015}).
This follows from the fact that for simple $k\frS_d$-modules
of dimension $\geq 2$ the Kronecker product is never simple (\cite[Main Theorem]{BK2000}). There are only two
$k\frS_d$-modules with dimension $1$, namely, by setting $\omega=(1,\dots,1)$, $L(\omega)=\sgn$, and $M^{(d)}=k$.
Now, $\F(L_\omega)=\sgn$ and $\F(Q^d)=k$.
Thus, the only cases where the tensor product might be simple are $\Lambda^d\otimes L_\mu$ and  $Q^d\otimes L_\mu$.

Consider first the case  $Q^d\otimes L_\mu$. By Lemma \ref{Lem:con_simple_adj}, this is the same as $\G_{\otimes}\F(L_{\mu})$
and by Lemma \ref{Lem:projcovervsExt} it is simple if and only if 
the top of $\rad P_\mu$ has only $p$-restricted composition factors.

For $\Lambda^d\otimes L_\mu$ use Lemma \ref{Lem:con_simple_adj} and  Lemma \ref{Lem:projcovervsExt} to obtain that 
$\Lambda^d\otimes L_\mu\cong\G_{\otimes}\F(L_{m(\mu)})$ is simple if and only if  all
$\nu$ with  $\Ext^1(L_{m(\mu)},L_\nu)\neq 0$ are $p$-restricted.
\end{proof}
It is not known in general when  $\Ext^1(L_\mu,L_\nu)\neq 0$ for partitions $\mu,\nu\in\Lambda(n,d)$,
so the question of when the internal tensor product of two simple functors is again simple
is not yet answered completely. Also the computation of the Mullineux map $m$ is not easy in general. 

\begin{corollary}\label{Cor:pcores}
 If $\mu$ is a $p$-core, then  $Q^d\otimes L_\mu\cong L_{\mu}$ and $\Lambda^d\otimes L_{m(\mu)}\cong L_{\mu}$.
\end{corollary}

\begin{proof}
If $\mu$ is a $p$-core, then it is the only simple in its block, i.e.\
$P_\mu=L_\mu$ and thus by \cite[3.2 Corollary]{DEN2004} $\Lambda^d\otimes L_{m(\mu)}\cong Q^d\otimes L_\mu\cong L_{\mu}$.
\end{proof}

\subsection{Special case}
In the case $n=d$ we make use of the following result to obtain some partitions $\mu$ such that
the tensor product $\Lambda^d\otimes L_\mu$ resp.\ $Q^d\otimes L_\mu$ is simple:
\begin{proposition}\cite[5.6 Proposition]{DEN2004}\label{Prop:DEN5.6}
Let $n=d$ and $p > 2$. Assume $\mu$ is $p$-restricted such that all $\lambda$ with $m(\mu') \geq \lambda$ are $p$-restricted. Then
\[\G_{\otimes}\F(L_{\mu})\cong L_\mu\]
\end{proposition}
In particular, every partition $\lambda$ such that all smaller partitions are $p$-restricted, provides a partition $\mu=(m(\lambda))'$
such that $\G_{\otimes}\circ\F(L_{\mu})= L_\mu$. Starting with the partition $(1,\dots,1)$ and
going through the elements of $\Lambda^+_p(n,d)$ in the lexicographic order, the first partitions are always all $p$-restricted.
The smallest not $p$-restricted partition $\nu$ is $\nu:=(p+1,1\dots,1)$ if $d\geq p+1$ and $\nu:=(d)$ otherwise. Thus,
for every $\mu=(m(\lambda))'$ with $\lambda< \nu$ we get $Q^d\otimes L_\mu\cong \G_{\otimes}\F(L_{\mu})\cong L_\mu$.

Unfortunately, the Proposition only provides a sufficient condition,
so one does not know what happens for partitions $\lambda$ such that $\lambda>\nu$.

\subsection*{The case $n=p=d$.} We can provide a full answer if we suppose in addition that $n=p$.
I am very thankful to Karin Erdmann for her advice regarding the Mullineux map in this case as well
as pointing out several composition series used in the proof of the following theorem.

\begin{theorem}\label{n=d=p}
  Let $k$ be a field of characteristic $p$ and $n=p=d>2$.
  The tensor
 product $L_\lambda\otimes L_\mu$ is simple if and only if, up to interchanging $\lambda$ and $\mu$,
 \begin{compactitem}
  \item[-] $\lambda=(1,\dots,1)$ and $\mu\neq(3,1^{p-3})$, or
  \item[-] $\lambda=(p-1,1)$ and  $\mu\neq (p-1,1)$.
 \end{compactitem}
 In these cases $L_{(1,\dots,1)}\otimes L_\mu\cong L_{m(\mu)}$ and $L_{(p-1,1)}\otimes L_\mu\cong L_{\mu}$.
\end{theorem}

\begin{proof}
We always have $\Lambda^d=L_{(1,\dots,1)}$ and, if $n=d=p$, the truncated symmetric powers $Q^d$ is 
the simple module indexed by the partition $(p-1,1)$, i.e.\ $Q^d=L_{(p-1,1)}$. 
Thus, by Theorem \ref{Th:simples}, all tensor products where $\lambda\neq (1,\dots,1)$ and  $\lambda\neq(p-1,1)$
are not simple. It remains to check the cases where $\lambda=(1,\dots,1)$ or $\lambda=(p-1,1)$.

Now all partitions $\mu$ not of the form $(p-k,1^k)$ for $1\leq k\leq p$ are $p$-cores, so in these cases
by Corollary \ref{Cor:pcores} 
\[L_{(p-1,1)}\otimes L_\mu\cong L_{\mu} \quad\text{ and }\quad
L_{(1,\dots,1)}\otimes L_\mu\cong L_{m(\mu)}.\]

Suppose now $\mu=(p-k,1^k)$. There is only one not  $p$-restricted partition, namely the partition $(p)$. 
We have $m((2,1^{p-2}))=(p)$, thus all but the partition $\mu=(p-1,1)$ fulfill the condition of Proposition \ref{Prop:DEN5.6}
and we get
\[L_{(p-1,1)}\otimes L_\mu\cong\G_{\otimes}\F(L_{\mu})\cong L_\mu\]
for all $\mu=(p-k,1^k)$ with $1<k\leq p$.
Since $m((p-1,1))=(3,1^{p-3})$ we also get
\[L_{(1,\dots,1)}\otimes L_\mu\cong\G_{\otimes}\F(L_{m(\mu)})\cong L_{m(\mu)}\]
for all $\mu=(p-k,1^k)$ with $1\leq k<3$ or $3<k\leq p$.

It remains the last cases $\mu=(p-1,1)$ resp.\ $\mu=(3,1^{p-3})$.
We know that 
\[S^{(p)}=\begin{array}{c}
            L_{(p-1,1)} \\
             L_{(p)}\\
        \end{array}\]
so that $S^{(p)}/\rad(S^{(p)})=L_{(p-1,1)}$ and hence there exists a surjection $P_{(p-1,1)}\twoheadrightarrow S^{(p)}$.
But then $L_{(p)}$ is in the top of $\rad(P_{(p-1,1)})$ and thus $\G_{\otimes}\F(L_{(p-1,1)})\cong L_{(p-1,1)}\otimes L_{(p-1,1)}$ is not simple.
Since $m((p-1,1))=(3,1^{p-3})$, we also get $L_{(1,\dots,1)}\otimes L_{(3,1^{p-3})}\cong L_{(p-1,1)}\otimes L_{(p-1,1)}$ is not simple.
\end{proof}

\end{document}